\documentclass[oneside, a4paper,reqno]{amsart}

\usepackage{pdfsync}
\usepackage{stmaryrd}
\usepackage{mathrsfs}
\usepackage{amsmath, amsthm, amscd, amssymb, latexsym, eucal}
\usepackage[all]{xy}
 \calclayout \makeatletter
 \calclayout \makeatletter
\def\serieslogo@{} \def\@setcopyright{} \makeatother
\makeatletter
\renewcommand*\env@matrix[1][c]{\hskip -\arraycolsep
  \let\@ifnextchar\new@ifnextchar
  \array{*\c@MaxMatrixCols #1}}
\makeatother

\numberwithin{equation}{section}
\newtheorem{thm}{Theorem}[section]
\newtheorem{cor}[thm]{Corollary}
\newtheorem{lem}[thm]{Lemma}
\newtheorem{prop}[thm]{Proposition}

\theoremstyle{definition}
\newtheorem{defn}[thm]{Definition}
\newtheorem{rem}[thm]{Remark}
\newtheorem{exam}[thm]{Example}

\newtheorem*{ackn}{Acknowledgments}

\newcommand{\lxr}{\longrightarrow}

\newcommand{\A}{\mathscr A}
\newcommand{\B}{\mathscr B}
\newcommand{\C}{\mathscr C}

\newcommand{\G}{\mathcal G}
\newcommand{\Hcal}{\mathcal H}

\newcommand{\X}{\mathcal X}
\newcommand{\Y}{\mathcal Y}
\newcommand{\Z}{\mathcal Z}

\newcommand{\Acal}{\mathscr A}
\newcommand{\Bcal}{\mathscr B}
\newcommand{\Ccal}{\mathscr C}
\newcommand{\Xcal}{\mathcal X}
\newcommand{\Ycal}{\mathcal Y}
\newcommand{\Zcal}{\mathcal Z}

 \DeclareMathOperator{\inc}{\mathsf{inc}}
\DeclareMathOperator*{\Ker}{\mathsf{Ker}}
 \DeclareMathOperator*{\Image}{\mathsf{Im}}
\DeclareMathOperator*{\Coker}{\mathsf{Coker}}

\DeclareMathOperator*{\Mod}{\mathsf{Mod}-\!}

\DeclareMathOperator*{\End}{\mathsf{End}}

\DeclareMathOperator{\Hom}{\mathsf{Hom}}

\DeclareMathOperator*{\Ext}{\mathsf{Ext}}
\DeclareMathOperator*{\Tor}{\mathsf{Tor}}

\newcommand{\iden}{\operatorname{Id}\nolimits}

\newsavebox{\proofbox}
\savebox{\proofbox}{\begin{picture}(7,7)%
  \put(0,0){\framebox(7,7){}}\end{picture}}

\begin{document}

\title{Recollements of Module Categories}

\author{Chrysostomos Psaroudakis, Jorge Vit\'oria}
\address{Department of Mathematics, University of Ioannina, 45110
Ioannina, Greece} \email{hpsaroud@cc.uoi.gr}
\address{Fakult\"at f\"ur Mathematik, Universit\"at Bielefeld, Postfach 10 01 31, D-33501, Bielefeld, Germany}
\email{jvitoria@math.uni-bielefeld.de}
\thanks{The first named author is co-funded by the European Union - European Social Fund (ESF) and National Sources, in the framework of the program ``HRAKLEITOS II" of the "Operational Program Education and Life Long Learning" of the Hellenic Ministry of Education, Life Long Learning and religious affairs. The second named author was supported by DFG  - SPP 1388, at the University of Stuttgart, for most of this project. This work was completed during a visit of the second author to the SFB 701 in Bielefeld.
}

\begin{abstract}
We establish a correspondence between recollements of abelian categories up to equivalence and certain \textsf{TTF}-triples. For a module category we show, moreover, a correspondence with idempotent ideals, recovering a theorem of Jans. Furthermore, we show that a recollement whose terms are module categories is equivalent to one induced by an idempotent element, thus answering a question by Kuhn.
\end{abstract}
\maketitle

\section{Introduction}

A recollement of abelian categories is an exact sequence of abelian categories 
where both the inclusion and the quotient functors admit left and right adjoints. 
They first appeared in the construction of the category of perverse sheaves on a singular space by Beilinson, Bernstein and Deligne (\cite{BBD}), 
arising from recollements of triangulated categories with additional structures (compatible t-structures). 
Properties of recollements of abelian categories were more recently studied by Franjou and Pirashvilli in \cite{Pira}, motivated by the MacPherson-Vilonen construction for the category of perverse sheaves (\cite{MV}). 

Recollements of abelian categories were used by Cline, Parshall and Scott to study module categories of finite dimensional algebras over a field  (see \cite{PS}). Later, Kuhn used them in the study of polynomial functors (\cite{Kuhn}), which arise not only in representation theory but also in algebraic topology and algebraic K-theory.  Recollements of triangulated categories have appeared in the work of Angeleri H\"ugel, Koenig and Liu in connection with tilting theory, homological conjectures and stratifications of derived categories of rings (\cite{AKL}, \cite{AKL2}, \cite{AKL3}). In particular, Jordan-H\"older theorems for recollements of derived module categories were obtained for some classes of algebras (\cite{AKL2}, \cite{AKL3}). Also, Chen and Xi have investigated recollements in relation with tilting theory (\cite{Xi:1}) and algebraic K-theory (\cite{Xi:2}).  Homological properties of recollements of abelian and triangulated categories have also been studied in \cite{recol}. 

Recollements and \textsf{TTF}-triples of triangulated categories are well-known to be in bijection (\cite{BBD}, \cite{Neeman}, \cite{Nicolas}). 
We will show that such a bijection holds for $\Mod{A}$ (see Proposition \ref{classif Mod}), where $\Mod{A}$ denotes the category of right $A$-modules, for a unitary ring $A$. Similar considerations in $\Mod{A}$ were made for {\em split} \textsf{TTF}-triples in \cite{NS}.  More generally, we show that recollements of an abelian category $\A$ (up to equivalence) are in bijection with bilocalising \textsf{TTF}-classes (see Theorem \ref{classif}).

Examples of recollements are easily constructed for the module category of triangular matrix rings (see  \cite{CZ}, \cite{GP}, \cite{LW}) or, more generally, using idempotent elements of a ring (see Example \ref{example}). In fact, we will see that there is a correspondence between idempotent ideals of $A$ and recollements of $\Mod{A}$, recovering Jans' bijection (\cite{Jans}) between \textsf{TTF}-triples in $\Mod{A}$ and idempotent ideals. Moreover, Kuhn conjectured in \cite{Kuhn} that if the categories of a recollement are equivalent to categories of modules over finite dimensional algebras over a field, then it is equivalent to one arising from an idempotent element. 
In our main result, we prove this conjecture for general rings. \\ 

\noindent\textbf{Theorem} [Theorem \ref{bijection Mod}, Corollary $5.5$] 
\textit{A recollement of $\Mod{A}$ is equivalent to a recollement in which the categories involved are module categories if and only if it is equivalent to a recollement induced by an idempotent element of a ring $S$, Morita equivalent to $A$. If, furthermore, $A$ is semiprimary, then any recollement of $\Mod{A}$ is equivalent to a recollement induced by an idempotent element of $A$.}\\

This paper is structured as follows. Section $2$ collects some preliminaries on recollements, (co)localisations, \textsf{TTF}-triples and ring epimorphisms. 
In section 3, we discuss \textsf{TTF}-triples in abelian categories and we use them in section 4 to classify recollements of abelian categories. Finally, in section 5 we focus on recollements of module categories, proving Kuhn's conjecture.

\begin{ackn}
This work was developed during a visit, funded by DFG - SPP 1388, of the first author to the University of Stuttgart in January $2013$. The first author would like to thank Steffen Koenig for the invitation and the warm hospitality. Both authors wish to express their gratitude to Apostolos Beligiannis and Steffen Koenig for valuable suggestions and comments. 
\end{ackn}

\section{Preliminaries}
Throughout, $\A$ denotes an abelian category. All subcategories considered are strict. For an additive  functor  $F$ between additive categories, we denote by $\Image F$ its essential image and by $\Ker F$ its kernel.

\subsection{\textsf{TTF}-triples}
A \textbf{torsion pair} in $\A$ is a pair $(\X,\Y)$ of full subcategories satisfying$\colon$ 
\begin{itemize}
\item $\Hom_{\A}(\X,\Y)=0$, i.e. $\Hom_{\A}(X,Y)=0$ $\forall X\in
\X$, $\forall Y\in \Y$;

\item For every object $A\in \A$, there are objects $X_A$ in $\X$ and $Y^A$ in $\Y$ and a short exact sequence $0\lxr X_A\lxr A\lxr Y^A\lxr 0.$
\end{itemize}
Given a torsion pair $(\X,\Y)$ in $\Acal$, we say that $\X$ is a \textbf{torsion class} and
$\Y$ is a \textbf{torsion-free class}. It follows easily from the definition that 
$$\X={^\circ{\Y}}:=\{A\in\Acal: \Hom_\Acal(A,Y)=0, \forall Y \in \Ycal\},$$ $$\Y={{\X}^\circ}:=\{A\in\Acal: \Hom_\Acal(X,A)=0, \forall X \in \Xcal\}$$ 
and that the assignment $\textsf{R}_{\X}(A)=X_A$ (respectively, $\textsf{L}_{\Y}(A)=Y^A$) yields an additive functor $\textsf{R}_{\X}\colon \A\lxr \X$ (respectively, $\textsf{L}_{\Y}\colon \A\lxr \Y$) which is right (respectively, left) adjoint of the inclusion functor $\textsf{i}_{\X}\colon \X\lxr \A$ (respectively, $\textsf{i}_{\Y}\colon \Y\lxr \A$). Hence, $\Xcal$ (respectively, $\Ycal$) is a reflective (respectively, coreflective) subcategory of $\Acal$.  Moreover, the endofunctors $\textsf{i}_\Xcal\textsf{R}_\Xcal$ and $\textsf{i}_\Ycal\textsf{L}_\Ycal$ satisfy$\colon$ 
\begin{itemize}
\item $\textsf{i}_\Xcal\textsf{R}_\Xcal$ is a \textbf{radical functor}, i.e., there is $\mu\colon \textsf{i}_\Xcal\textsf{R}_\Xcal\lxr \iden_{\A}$ a natural transformation such that $\mu_A$ is a monomorphism and $\textsf{i}_\Xcal\textsf{R}_\Xcal(\Coker\mu_A)=0$. 
\item $\textsf{i}_\Ycal\textsf{L}_\Ycal$ is a \textbf{coradical functor}, i.e., there is $\nu\colon \iden_{\A}\lxr \textsf{i}_\Ycal\textsf{L}_\Ycal$ a natural transformation such that $\mu_A$ is an epimorphism and $\textsf{i}_\Ycal\textsf{L}_\Ycal(\Ker\nu_A)=0$. 
\item Both $\textsf{i}_\Xcal\textsf{R}_\Xcal$ and $\textsf{i}_\Ycal\textsf{L}_\Ycal$ are \textbf{idempotent}, i.e., both $\mu_{\textsf{i}_\Xcal\textsf{R}_\Xcal(A)}$ and $\nu_{\textsf{i}_\Ycal\textsf{L}_\Ycal(A)}$ are isomorphisms, for all $A$ in $\Acal$.
\end{itemize}
In fact, there are bijections between torsion pairs in $\A$, idempotent radical  functors $F\colon \A\lxr \A$ and idempotent coradical functors $G\colon \A\lxr \A$. Thus, the endofunctors $\textsf{i}_\Xcal\textsf{R}_\Xcal$ and $\textsf{i}_\Ycal\textsf{L}_\Ycal$ determine the torsion pair uniquely (\cite[Theorem $2.8$]{Dickson}, \cite[Theorem $1.2$]{Col}). 
Often, torsion and torsion-free classes can be identified by closure properties. Recall that $\A$ is said to be \textbf{well-powered} if  the class of subobjects of any given object forms a set.

\begin{prop}\cite[Theorem $2.3$]{Dickson}\label{Dickson}
Let $\Acal$ be a well-powered, complete and cocomplete abelian category. A full subcategory $\X$ is a torsion (respectively, torsion-free) class if and only if it is closed under quotients, extensions and coproducts (respectively, subobjects, extensions and products).
\end{prop}

Recall that a torsion pair $(\X,\Y)$ in $\A$ is \textbf{hereditary} if $\X$ is closed under subobjects and \textbf{cohereditary} if $\Y$ is closed under quotients. We will be interested in classes which are both torsion and torsion-free.

\begin{defn}
A triple $(\X,\Y,\Z)$ of full subcategories of $\A$ is called a \textbf{\textsf{TTF}-triple} (and $\Y$ is a {\bf \textsf{TTF}-class}) if $(\X,\Y)$ and $(\Y,\Z)$ are torsion pairs.
\end{defn}
 
Clearly, if $(\X,\Y,\Z)$ is a \textsf{TTF}-triple in $\A$, then the torsion pair $(\X,\Y)$ is cohereditary and $(\Y,\Z)$ is hereditary. By Proposition \ref{Dickson}, when $\A$ is well-powered, complete and cocomplete, a full subcategory $\Y$ of $\A$ is a \textsf{TTF}-class if and only if it is closed under products, coproducts, extensions, subobjects and quotients. We refer to \cite{BR} for further details on torsion theories and \textsf{TTF}-triples in both abelian and triangulated categories. 
In ring theory, \textsf{TTF}-triples are well understood due to the following result of Jans, which  will be proved in section $5$ using our results on \textsf{TTF}-triples of abelian categories. 

\begin{thm}\cite[Corollary $2.2$]{Jans}\label{Jans}
There is a bijection between \textsf{TTF}-triples in $\Mod{A}$ and idempotent ideals of the ring $A$.
\end{thm}

\subsection{Localisations and Colocalisations}
For a subcategory $\Y$ of $\A$ closed under subobjects, quotients and extensions, Gabriel constructed in \cite{Gabriel} an abelian category $\A/\Y$ with morphisms 
\begin{equation}\label{homs}
\Hom_{\A/\Y}(j^*(M),j^*(N))=\varinjlim_{\stackrel{N'\leq N: N'\in\Ycal}{M'\leq M:M/M'\in\Ycal}}\Hom_\Acal(M',N/N').
\end{equation}
Such a subcategory $\Y$ is called a \textbf{Serre subcategory} and it yields an exact and dense \textbf{quotient functor} $j^*\colon \Acal \lxr \Acal/\Ycal$. A Serre subcategory $\Y$ is said to be \textbf{localising} (respectively, \textbf{colocalising}) if $j^*$  admits a right (respectively, left) adjoint. Moreover, it is said to be \textbf{bilocalising} if it is both localising and colocalising.
These properties are related to the structure of subcategories orthogonal to $\Y$ with respect to the pairings $\Hom_\Acal(-,-)$ and $\Ext^1_\Acal(-,-)$ (in the sense of Yoneda), i.e., 
$${^\perp{\Y}}:=\{A\in\Acal: \Hom_\Acal(A,Y)=0={\Ext}_\Acal^1(A,Y), \forall Y\in \Ycal\} \ \ \text{and} $$
$${{\Y}^\perp}:=\{A\in\Acal: \Hom_\Acal(Y,A)=0={\Ext}_\Acal^1(Y,A), \forall Y\in \Ycal\}. \ \ \ \ \ $$

\begin{thm}\cite[Lemma $2.1$, Proposition $2.2$]{Geigle-Lenzing}\label{loc coloc}
The following hold for a Serre subcategory $\Ycal$ of $\Acal$. 
\begin{enumerate}
\item The quotient functor $j^*$ induces fully faithful functors ${{\Y}^\perp}\lxr \Acal/\Ycal$ and ${^\perp{\Y}}\lxr \Acal/\Ycal$.
\item The functor $j^*\colon{{\Y}^\perp}\lxr \Acal/\Ycal$ is an equivalence if and only if $\Ycal$ is localising, in which case a quasi-inverse for $j^*$ is its right adjoint $j_*$.
\item The functor $j^*\colon{^\perp{\Y}}\lxr \Acal/\Ycal$ is an equivalence if and only if $\Ycal$ is colocalising, in which case a quasi-inverse for $j^*$ is its left adjoint $j_!$.
\end{enumerate}
\end{thm}

Localisations and colocalisations with respect to a torsion pair $(\Xcal,\Ycal)$ in $\Acal$ 
first appeared in \cite{Gabriel} (see also \cite{Col}, \cite{St}). 
As in \cite{Col}, we say that $(\X,\Y)$ is \textbf{strongly hereditary}, (respectively \textbf{strongly cohereditary}), if there is a functor $\textsf{L}\colon\Acal\lxr\Acal$ (respectively, $\textsf{C}\colon\Acal\lxr\Acal$), the \textbf{localisation}, (respectively the \textbf{colocalisation}) \textbf{functor with respect to $(\X,\Y)$}, and a natural transformation $\phi\colon \iden_\Acal\lxr \textsf{L}$ (respectively, $\psi\colon \textsf{C}\lxr \iden_\Acal$) such that, for all $A$ in $\A$$\colon$
\begin{enumerate}
\item $\Ker\phi_A, \Coker\phi_A\in\X$ (respectively, $\Ker\psi_A, \Coker\psi_A\in\Y$);

\item $\textsf{L}(A)\in\Y$ (respectively, $\textsf{C}(A)\in\X$);

\item $\textsf{L}(A)$ is \textbf{$\X$-divisible} (respectively, $\textsf{C}(A)$ is \textbf{$\Y$-codivisible}), meaning that $\Hom_\Acal(-,\textsf{L}(A))$ (respectively, $\Hom_\Acal(\textsf{C}(A),-)$) is exact on exact sequences, $0\rightarrow K\rightarrow M\rightarrow N\rightarrow 0$, with $N\in\X$ (respectively, $K\in\Y$).
\end{enumerate}
The embedding in $\A$ of $\Image \textsf{L}$, the \textbf{Giraud subcategory} of $\A$ associated with $(\X,\Y)$, admits an exact left adjoint such that $\textsf{L}$ is given by the composition of the functors and $\phi$ is the unit of this adjunction (\cite{Col}). Also, $\Image \textsf{L}$ is the full subcategory of $\X$-divisible objects of $\Y$. Dual statements holds for $\Image \textsf{C}$. 
\begin{defn}
We say that a \textsf{TTF}-triple $(\X,\Y,\Z)$ in $\Acal$ is \textbf{strong} if $(\X,\Y)$ is strongly cohereditary and $(\Y,\Z)$ is strongly hereditary.
\end{defn}

If $\A$ has enough projectives (respectively, injectives), then by \cite[Theorem 1.8-1.8$^*$]{Col}, a torsion pair is cohereditary (respectively, hereditary) if and only if it is strongly cohereditary (respectively, strongly hereditary)

\subsection{Recollements} We now discuss recollements of abelian categories (\cite{BBD, Pira, Kuhn}).

\begin{defn}\label{rec}
A \textbf{recollement} of an abelian category $\Acal$ by abelian categories $\B$ and $\C$, denoted by $\textsf{R}(\B,\A,\C)$, is a diagram of additive functors as follows,   satisfying the conditions below.
\[\textsf{R}(\B,\A,\C)\colon\ \ \ \ \  \xymatrix@C=0.5cm{
\B \ar[rrr]^{i_*} &&& \A \ar[rrr]^{j^*}  \ar @/_1.5pc/[lll]_{i^*}  \ar
 @/^1.5pc/[lll]_{i^!} &&& \C
\ar @/_1.5pc/[lll]_{j_!} \ar
 @/^1.5pc/[lll]_{j_*}
 } 
 \]
\begin{enumerate}
\item $(j_!,j^*,j_*)$ and $(i^*,i_*,i^!)$ are adjoint triples;

\item The functors $i_*$, $j^!$, and $j_*$ are fully faithful;

\item $\Image{i_*}=\Ker{j^*}$. 
\end{enumerate}
\end{defn}

Throughout, we fix a recollement $\textsf{R}(\B,\A,\C)$ as in Definition \ref{rec}. 
The next proposition collects some properties of $\textsf{R}(\B,\A,\C)$ that can be easily derived from the definition (see for example \cite{Pira}, \cite{recol}).

\begin{prop}\label{properties}
The following hold for a recollement $\textsf{R}(\B,\A,\C)$.
\begin{enumerate}
\item The functors $i_*$ and $j^*$ are exact.
\item $i^*j_!=0=i^!j_*$, $j^*j_!=\iden_\Ccal=j^*j_*$ and $i^!i_*=\iden_\Bcal=i^*i_*$.
\item $\Bcal$ is a Serre subcategory of $\Acal$ and $j^*$ is naturally equivalent to the quotient functor $\Acal\lxr \Acal/\Bcal$. In particular, we have that $\Ccal\cong \Acal/\Bcal$ and that  $\Bcal$ is bilocalising.
\item For all $A$ in $\Acal$, there are $B$ and $B'$ in $\Bcal$ such that the units and counits of the adjunctions induce the following exact sequences
$$0\rightarrow i_*(B)\rightarrow j_!j^*(A)\rightarrow A\rightarrow i_*i^*(A)\rightarrow 0,$$ $$0\rightarrow i_*i^!(A)\rightarrow A\rightarrow j_*j^*(A)\rightarrow i_*(B')\rightarrow 0.$$
\end{enumerate}
\end{prop}

A full subcategory of $\Acal$ is said to be \textbf{bireflective} if it is reflective and coreflective. 
For a recollement $\textsf{R}(\B,\A,\C)$, $i_*(\B)$ is a bireflective subcategory of $\A$. 

\begin{rem}\label{rem Ker}
Any \textsf{TTF}-class is clearly bireflective. Also, any bilocalising subcategory $\Ycal$ is bireflective and, thus, it induces a recollement of $\Acal$. Indeed, let $j_!$ be the left adjoint and $j_*$ be the right adjoint of the quotient functor $j^*\colon\Acal\lxr\Acal/\Ycal$. We can define functors $i^*\colon \Acal\lxr \Ycal$ and $i^!\colon\Acal\lxr \Ycal$ by setting $i^*(A)$ to be the cokernel of the counit of the adjunction $(j_!,j^*)$ at $A$ and $i^!(A)$ the kernel of the unit of the adjunction $(j_*,j^*)$ at $A$. It then follows that  $i^*$ is a left adjoint  and $i^!$ is a right adjoint of the inclusion $i_*\colon\Ycal\lxr\Acal$. 
\end{rem}

We end this subsection with a widely studied example of a recollement. 

\begin{exam}\cite{Kuhn}\cite{PS}\cite{recol}\label{example}
Let $A$ be a ring and $e$ an idempotent element of $A$. There is
a recollement $\textsf{R}(\Mod{A/AeA},\Mod{A},\Mod{eAe})$ of $\Mod{A}$, as in the diagram below, which is said to be \textbf{induced by the idempotent element $e$}. 
\[
\xymatrix@C=0.5cm{
\Mod{A/AeA} \ar[rrr]^{\inc} &&& \Mod{A} \ar[rrr]^{\Hom_A(eA,-) \ } \ar
@/_1.5pc/[lll]_{-\otimes_AA/AeA}  \ar
 @/^1.5pc/[lll]^{\Hom_A(A/AeA,-)} &&& \Mod{eAe}.
\ar @/_1.5pc/[lll]_{-\otimes_{eAe}eA} \ar
 @/^1.5pc/[lll]^{\Hom_{eAe}(Ae,-)}
 } 
 \]
\end{exam}

\subsection{Ring Epimorphisms}
We fix $A$ a unitary ring. To study recollements of $\Mod{A}$ we look at its bireflective subcategories, which are classified by epimorphisms in the category of unitary rings. A ring homomorphism $f\colon A\lxr B$ is an epimorphism if and only if the restriction functor $f_*\colon \Mod{B}\lxr \Mod{A}$ is fully faithful (\cite{St}).  Theorem \ref{Gab de la} states that all bireflective subcategories of $\Mod{A}$ arise in this way. Two ring epimorphisms $f\colon A\lxr B$ and $g\colon A\lxr C$ lie in the same \textbf{epiclass of $A$}, if there is a ring isomorphism  $h\colon B\lxr C$ such that $g=hf$.

\begin{thm}\cite[Theorem 1.6.3]{I}\cite{Geigle-Lenzing}\cite[Theorem 1.2]{GdP}\label{Gab de la}   There is a bijection between epiclasses of $A$ and bireflective subcategories of $\Mod{A}$, defined by assigning to an epimorphism $f\colon A\lxr B$, the subcategory $\Xcal_B:=\Image{f_*}$. Moreover, a full subcategory $\Xcal$ of $\Mod{A}$ is bireflective if and only if it is closed under products, coproducts,  kernels and cokernels.
\end{thm}

\begin{rem}\label{Tor}
Some properties of a ring epimorphism $f\colon A\lxr B$ can be seen from the bireflective subcategory $\Xcal_B$. For example, $\Xcal_B$ is extension-closed if and only if $\Tor_1^A(B,B)=0$ (\cite{Geigle-Lenzing}).
\end{rem}

Given a ring epimorphism $f\colon A\lxr B$, let $\psi_M\colon M\lxr M\otimes_A B$ denote the unit of the adjoint pair $(-\otimes_AB,f_*)$ at a right $A$-module $M$ (given by $\psi_M(m)=m\otimes 1_B$, for all $m$ in $M$). Note that $\psi_N$ is an isomorphism for all $N$ in $\Xcal_B$. In fact, $\psi_M$ is the $\Xcal_B$-reflection of the right $A$-module $M$. In particular, $f\colon A\lxr B$, regarded as a morphism in $\Mod{A}$, is the $\Xcal_B$-reflection $\psi_A$. 

\section{\textsf{TTF}-triples in abelian categories}
In this section we discuss some aspects of \textsf{TTF}-triples in abelian categories. 

We start with an adaptation of the classical bijection between torsion pairs and idempotent radicals. This will, later, yield a proof for Jans' correspondence (Theorem \ref{Jans}). 

\begin{prop}\label{gen Jans}
Let $\Acal$ be a well-powered, complete and cocomplete abelian category $\Acal$. There are bijections between the following classes$\colon$
\begin{enumerate}
\item \textsf{TTF}-triples in $\Acal$;
\item Left exact radical functors $F\colon \Acal\lxr\Acal$ preserving products;
\item Right exact coradical functors $G\colon \Acal\lxr\Acal$ preserving coproducts.
\end{enumerate}
\end{prop}
\begin{proof}
We show a bijection between (i) and (ii) (a bijection with (iii) can be obtained dually). Let $(\X,\Y,\Z)$ be a \textsf{TTF}-triple in $\Acal$. 
By Remark \ref{rem Ker}, $\Y$ is bireflective and, hence, $\textsf{i}_\Ycal$ of $\Y$ in $\A$ admits a left adjoint $\textsf{L}_\Y$ and a right adjoint $\textsf{R}_\Ycal$. Thus, $\textsf{i}_\Ycal$ is exact and $\textsf{i}_{\Y} \textsf{R}_{\Y}$ is a left exact radical functor preserving products. We define a correspondence$\colon$
\[
\xymatrix{
  \Phi\colon (\X,\Y,\Z) \ \ar@{|->}[r] & \ ( \textsf{i}_{\Y}\textsf{R}_{\Y}\colon \A\lxr \A ).    }
\]
Given a left exact radical functor $F\colon \Acal\lxr \Acal$ preserving products, it is easy to see that $F$ is idempotent and $(\Ycal_F:=\{A\in \A \ | \ F(A)=A \},\Y^\circ_F=\Ker F)$ is a hereditary torsion pair (see also \cite[Proposition VI.1.7]{St}). Since $F$ preserves products, $\Y_F$ is closed under products and thus, Proposition \ref{Dickson} shows that $\Y_F$ is a \textsf{TTF}-class. Hence, we can associate a \textsf{TTF}-triple to $F$ as follows.
\[
\xymatrix{
\Psi\colon F \ \ar@{|->}[r] & \  
 (^\circ{\Y_F},\Y_F,{\Y_F}^\circ) }
\]
Finally, it easily follows that $\Phi$ and $\Psi$ are inverse correspondences.
\end{proof}

Now we will identify \textsf{TTF}-classes which are localising and colocalising.

\begin{lem}\label{loc coloc comparison}
Let $(\X,\Y,\Z)$ be a \textsf{TTF}-triple in $\Acal$. Then $\Ycal$ is a localising subcategory of $\Acal$ if and only if $(\Ycal,\Z)$  is a strongly hereditary torsion pair. Dually, $\Ycal$ is a colocalising subcategory if and only if $(\X,\Y)$ is a strongly cohereditary torsion pair.
\end{lem}
\begin{proof}
We prove the first statement (the second one is dual). 
Suppose that the torsion pair $(\Y,\Z)$ is strongly hereditary. Let $\textsf{L}$ and $\phi\colon\iden_\A\lxr \textsf{L}$ be the associated localisation functor and natural transformation, respectively, and $j^*\colon\A\lxr \A/\Y$ the quotient functor. Recall that $\textsf{L}=\textsf{i}\textsf{l}$, where $\textsf{i}\colon \G\lxr \A$ is the inclusion functor of the Giraud subcategory $\G:=\Image \textsf{L}$ in $\A$ and $\textsf{l}$ is its left adjoint. We observe that $j^*\textsf{i}\colon\G\lxr \A/\Y$ is an equivalence. It is dense since it is easy to check that $j^*\phi$ is a natural equivalance between $j^*$ and $j^*\textsf{L}$.  On the other hand, it is fully faithful by the description (\ref{homs}) of morphisms in $\Acal/\Ycal$. Indeed, given $M$ and $N$ in $\G$, there are no subobjects of $N$ lying in $\Ycal$ and, for all subobjects $M'$ of $M$ such that $M/M'$ lies in $\Ycal$, $\Ycal$-divisibility guarantees that $\Hom_\Acal(M',N)=\Hom_\Acal(M,N)$. Since both $j^*\textsf{i}$ and $\textsf{l}$ have right adjoints then so does $(j^*\textsf{i})\textsf{l}\cong j^*\textsf{L}\cong j^*$.

Conversely, suppose that $\Ycal$ is a localising subcategory of $\Acal$ and let $j_*\colon \Acal/\Ycal\lxr \Acal$ be the right adjoint of the quotient $j^*\colon\Acal\lxr\Acal/\Ycal$. Given any object $A$ in $\Acal$, consider the map given by the unit of the adjunction $\phi_A\colon A\lxr j_*j^*(A)$. We will show that this is a localisation with respect to $(\Y,\Z)$. By Theorem \ref{loc coloc}, $j_*j^*(A)$ lies in $\Ycal^\perp$ and, thus, in $\Zcal=\Ycal^\circ$. Since $j^*$ is exact and $j^*j_*\cong \iden_{\Acal/\Ycal}$, it is also clear that $j^*(\Ker{\phi_A})=0=j^*(\Coker{\phi_A})$ and, thus, both $\Ker{\phi_A}$ and $\Coker{\phi_A}$ lies in $\Ycal$. Finally, since $\Ext^1_\Acal(Y,j_*j^*(A))=0$ for all $Y$ in $\Ycal$, $j_*j^*(A)$ is $\Y$-divisible, as wanted.
\end{proof}

\begin{rem}\label{Giraud, Co-Giraud}
In Lemma $3.2$ we in fact prove that, if $\Ycal$ is a localising subcategory and $j_*$ is the right adjoint of the quotient functor $j^*\colon \Acal\lxr \Acal/\Ycal$, then the Giraud subcategory $\G$ associated to the strongly hereditary torsion pair $(\Ycal,\Zcal)$ coincides with $\Ycal^\perp$. 
Similarly, the Co-Giraud subcategory $\Hcal$ of $\Acal$ induced by the strongly cohereditary torsion pair $(\X,\Y)$ (formed by the $\Y$-codivisible objects of $\X$) coincides with the subcategory $^\perp\Y$.
\end{rem}

\section{Recollements of abelian categories}
We define an equivalence relation on the class of recollements of $\Acal$. Although seemingly artificial, Lemma \ref{two eq} shows that  Definition \ref{def eq rec} is natural.

\begin{defn}\label{def eq rec}
Two recollements $\textsf{R}(\Bcal,\Acal,\Ccal)$ and $\textsf{R}(\Bcal',\Acal',\Ccal')$ 
are \textbf{equivalent} if there are equivalence functors $\Phi\colon \Acal\lxr \Acal'$ and $\Theta\colon \Ccal\lxr \Ccal'$ such that 
the diagram below commutes up to natural equivalence, i.e. there is a natural equivalence of functors between $\Theta j^*$ and $j^{*'}\Phi$.
\[
\xymatrix{
  \A \ar[d]_{\Phi}^{\simeq} \ar[r]^{j^{*}} & \C \ar[d]^{\Theta}_{\simeq}     \\
  \A'    \ar[r]^{j^{*'}} & \C'                  } 
\]
\end{defn}

\begin{lem}\label{two eq}
Two recollements $\textsf{R}(\Bcal,\Acal,\Ccal)$ and $\textsf{R}(\Bcal',\Acal',\Ccal')$ are equivalent if and only if there are exact equivalences $\Phi\colon\Acal\lxr\Acal'$, $\Psi\colon\Bcal\lxr\Bcal'$ and $\Theta\colon\Ccal\lxr\Ccal'$ such that the six diagrams associated to the six functors of the recollements commute up to natural equivalences.
\end{lem}
\begin{proof}
The condition in the lemma is clearly sufficient to get an equivalence of recollements. Conversely, suppose that we have an equivalence of recollements as in Definition \ref{def eq rec}. Recall that left (or right) adjoints of naturally equivalent functors are naturally equivalent. Thus, the left (or right) adjoints of $\Theta j^*$ and of $j^{*'}\Phi$ are equivalent. Such adjoints can be obtained by choosing a quasi-inverse of the equivalences $\Phi$ and $\Theta$. Using then the fact that the composition of two quasi-inverse functors is naturally equivalent to the identity functor, we easily get the desired natural equivalences between $\Phi j_!$ and $j_!^{'}\Theta$ and between $\Phi j_*$ and $j_*^{'}\Theta$. Up to equivalence, the two recollements are uniquely determined by these functors (see Remark \ref{rem Ker}). Let $\Psi$ be the restriction of $\Phi$ to $\Ker{j^*}$ (which is equivalent to $\Bcal$), where $j^*\colon \Acal\lxr\Ccal$. Then, the diagram associated with the inclusion functor $i_*\colon\Ker{j^*}\lxr\Acal$ clearly commutes and so do the other two, by an adjunction argument analogous to the one above.
\end{proof}

Equivalences of recollements whose outer equivalence functors ($\Psi$ and $\Theta$ in the lemma) are the identity functor have been studied in \cite{Pira}. Equivalences of recollements of triangulated categories also appear in \cite[Theorem 2.5]{PS}.

In the following theorem, we use the fact that structural properties of $\A$, such as $\textsf{TTF}$-triples, are preserved under equivalence.

\begin{thm}\label{classif}
Let $\Acal$ be an abelian category. The following are in bijection.
\begin{enumerate}
\item Equivalence classes of recollements of abelian categories $\textsf{R}(\Bcal,\Acal,\Ccal)$;
\item Strong \textsf{TTF}-triples $(\Xcal,\Ycal,\Zcal)$ in $\Acal$;
\item Bilocalising \textsf{TTF}-classes $\Ycal$ of $\Acal$;
\item Bilocalising Serre subcategories $\Ycal$ of $\Acal$.
\end{enumerate}
\end{thm}
\begin{proof}
Let $\textsf{R}(\B,\A,\C)$ be a recollement of  $\A$. Firstly, $(\Ker{i^*},i_*(\Bcal),\Ker{i^{!}})$ is a \textsf{TTF}-triple in $\A$. The adjoint triple $(i^*,i_*,i^!)$ ensures that $$\Hom_\Acal(\Ker{i^*},i_*(\Bcal))=0=\Hom_\Acal(i_*(\Bcal),\Ker{i^{!}}).$$
Let $A$ be an object of $\A$. From Proposition \ref{properties}, we have an exact sequence
\[
\xymatrix{
0 \ar[r] & \Ker{\mu_A} \ar[r] & j_!j^*(A) \ar[r]^{\ \ \ \mu_A}
& A \ar[r] & i_*i^*(A) 
  \ar[r]^{ } & 0 }
\]
where $\Ker{\mu_A}$ lies in $i_*(\Bcal)$. 
Applying the right exact functor $i^*$ to the sequence, we see that $i^*(\Image{\mu_A})=0$. Thus, the sequence $$0\lxr \Image{\mu_A}\lxr A\lxr i_*i^*(A)\lxr 0$$ shows that $(\Ker i^*,i_*(\Bcal))$ is a torsion pair. Similarly, the exact sequence induced by $\nu_A\colon A\lxr j_*j^*(A)$ can be used to show that $(i_*(\Bcal),\Ker{i^!})$ is a torsion pair in $\B$.
Since $\textsf{R}(\Bcal,\Acal,\Ccal)$ is a recollement, $i_*(\Bcal)$ is a bilocalising subcategory of $\Acal$. Hence, by Lemma \ref{loc coloc comparison}, the torsion pairs $(\Ker{i^*},i_*(\Bcal))$ and $(i_*(\Bcal),\Ker{i^!})$ are, respectively, strongly cohereditary and strongly hereditary and $(\Ker{i^*},i_*(\Bcal),\Ker{i^!})$ is a strong \textsf{TTF}-triple.  Note that this $\textsf{TTF}$-triple depends only on the equivalence class of the recollement. Indeed, if $\textsf{R}(\B',\A',\C')$ is a recollement equivalent to $\textsf{R}(\Bcal,\Acal,\Ccal)$ via an equivalence $\Phi\colon\A'\lxr \A$, then the corresponding $\textsf{TTF}$-class of $\A$ associated to it is given by $\Phi i_*'(\B')$ which coincides, by Lemma \ref{two eq}, with $i_*(\B)$.

We construct now an inverse correspondence (see also \cite[Theorem 4.5]{Col}). Let  $(\X,\Y,\Z)$ be a strong \textsf{TTF}-triple in $\A$. Since $\Y$ is a \textsf{TTF}-class, by Remark \ref{rem Ker} it is bireflective and the embedding $i_*$ of $\Y$ in $\A$ admits a left adjoint $i^*$ and a right adjoint $i^!$. It is also a Serre subcategory, and we consider the quotient functor $j^*\colon\Acal\lxr \Acal/\Ycal$. Since the triple is strong it follows from Lemma \ref{loc coloc comparison} that $\Y$ is bilocalising. Thus, $j^*$ has both left and right adjoints, $j_!$ and $j_*$ respectively, which are fully faithful (because $j^*j_*$ and $j^*j_!$ are naturally equivalent to $\iden_{\Acal/\Ycal}$, see \cite{Gabriel}). 
Hence, we have a recollement $\textsf{R}(\Ycal,\Acal,\Acal/\Ycal)$. Clearly these correspondences are inverse to each other, up to equivalence of recollements.

Finally, since $i_*(\Bcal)$ is a bilocalising \textsf{TTF}-class as well as a (bireflective) Serre subcategory, the bijection between (i) and (ii) easily implies the bijections between (i), (iii) and (iv).
\end{proof}

Under some conditions on $\Acal$, the above bijection becomes more clear.

\begin{cor}\label{cor enough}
If $\Acal$ has enough projectives and injectives, then the equivalence classes of recollements of $\Acal$ are in bijection with the \textsf{TTF}-triples in $\Acal$.
\end{cor}
\begin{proof}
Let $(\X,\Y,\Z)$ be a \textsf{TTF}-triple in $\A$. Since $\Acal$ has enough projectives and injectives it follows from \cite[Theorem $1.8/1.8^*$]{Col} that 
every \textsf{TTF}-triple in $\A$ is strong. The result then follows from Theorem \ref{classif}.
\end{proof}

Given a recollement, we then have the following notable equivalences.

\begin{cor}
Let $\textsf{R}(\Bcal,\Acal,\Ccal)$ be a recollement of $\Acal$, $\G$ be the Giraud subcategory associated to the torsion pair $(i_*(\Bcal),\Ker{i^!})$ and $\Hcal$ the Co-Giraud subcategory associated to the torsion pair $(\Ker{i^*},i_*(\Bcal))$. Then, $j^*$ induces$\colon$ 
\begin{enumerate}
\item an equivalence $\xymatrix{i_*(\B)^\perp=\G=\Image{j_*} \ \ar[r]^{\ \ \ \ \ \cong} & \ \Acal/i_*(\B)},$

\item an equivalence $\xymatrix{^\perp i_*(\B)=\Hcal=\Image{j_!} \ \ar[r]^{\ \ \ \ \ \cong} & \ \Acal/i_*(\B)}$ and

\item an equivalence $\xymatrix{\Ker{i^*}\cap\Ker{i^!} \ \ar[r]^{\ \ \ \ \cong} & \ \Acal/i_*(\B)}$.
\end{enumerate}
\end{cor}
\begin{proof}
Statements (i) and (ii) follow immediately from the fact that the \textsf{TTF}-triple $(\Ker{i^*},i_*(\Bcal),\Ker{i^!})$ is strong and from Remark \ref{Giraud, Co-Giraud}. Statement (iii) is well-known for \textsf{TTF}-triples (see \cite[Theorem $1.9$]{Gentle}). 
\end{proof}

\section{Recollements of Module Categories and \textsf{TTF}-triples}

In this section, $A$ is a unitary ring and $\Mod{A}$ the category of right $A$-modules. Since $\Mod{A}$ has enough projectives and injectives, by Corollary \ref{cor enough}, there is a bijection between equivalence classes of recollements of $\Mod{A}$ and \textsf{TTF}-triples in $\Mod{A}$. Moreover, there is a bijection between \textsf{TTF}-classes and bireflective Serre subcategories, since the closure conditions for both types of subcategories are the same (see Proposition \ref{Dickson} and Theorem \ref{Gab de la}). In particular, any bireflective Serre subcategory of $\Mod{A}$ is bilocalising by Lemma \ref{loc coloc comparison}. 
We will describe these categories in terms of ring epimorphisms. Similar results can be found in \cite[Section 7]{Aus} and in \cite[Proposition 5.3]{Geigle-Lenzing}.

\begin{prop}\label{Serre biref Mod}
Let $\Ycal$ be a bireflective Serre subcategory of $\Mod{A}$. Then there is an idempotent ideal $I$ of $A$ such that $\Ycal$ is the essential image of the restriction functor induced by the ring epimorphism $f\colon A\lxr A/I$.
\end{prop}
\begin{proof}
Since $\Y$ is a bireflective subcategory of $\Mod{A}$, by Theorem \ref{Gab de la}, there is a ring epimorphism $f\colon A\lxr B$, for some ring $B$, such that $f_*(\Mod{B})=\Y$. We will now prove that $f$ is surjective. Since $\Y$ is a \textsf{TTF}-class, $({^\circ{\Y}},\Y)$ is a torsion pair and the composition $\textsf{i}_\Y \textsf{L}_\Y$ ($\textsf{i}_\Y$ being the inclusion $\Y \lxr \Mod{A}$ and $\textsf{L}_\Y$ its left adjoint) is the idempotent coradical functor sending a module $M$ to its torsion-free part. 
In particular the unit of this adjunction is surjective on every $A$-module. On the other hand, since $f_*(\Mod{B})=\Y$, it follows that $\textsf{i}_\Y \textsf{L}_\Y$ is naturally equivalent to $f_*(-\otimes_AB)$. Thus, $\psi_M\colon M\lxr M\otimes_AB$ is surjective for every $A$-module $M$. In particular, $f=\psi_A$ is surjective. Since $f_*(\Mod{B})$ is closed under extensions, for $I=\Ker f$ we have by Remark \ref{Tor}
$$0={\Tor}_1^A(B,B)={\Tor}_1^A(A/I,A/I)=I/I^2$$
and, thus, we infer that $\Y=f_*(\Mod{A/I})$, with $I^2=I$.
\end{proof}

We now recover Jans' bijection between \textsf{TTF}-triples and
idempotent ideals (Theorem \ref{Jans}) and classifiy equivalence classes of recollements of $\Mod{A}$.

\begin{prop}\label{classif Mod}
There is a bijection between equivalence classes of recollements of $\Mod{A}$,  \textsf{TTF}-triples in $\Mod{A}$ and idempotent ideals of $A$.
\end{prop}
\begin{proof}
The bijection between equivalence classes of recollements and \textsf{TTF}-triples follows from Corollary \ref{cor enough}, since $\Mod{A}$ has enough projectives and injectives. The bijection between \textsf{TTF}-triples and idempotent ideals of $A$ can be seen as a consequence of Proposition \ref{gen Jans}. Indeed, the bijection in that proposition assigns to a \textsf{TTF}-triple a functor which, by Proposition \ref{Serre biref Mod}, is precisely $f_*(-\otimes_AA/I)$ for some idempotent ideal $I$ and $f\colon A\lxr A/I$ the canonical projection, thus uniquely determined by the ideal $I$. Conversely, given an idempotent ideal $I$ and the quotient map $f\colon A\lxr A/I$ it is easy to check that $f_*(-\otimes_AA/I)$ is a right exact idempotent coradical endofunctor of $\Mod{A}$ preserving coproducts, thus finishing the proof.
\end{proof}
We say that a recollement of $\Mod{A}$ is a \textbf{recollement by module categories} if it is equivalent to a recollement in which the categories involved are module categories. We  recall the conjecture made by Kuhn in \cite{Kuhn}.\\

\noindent\textbf{Conjecture} \cite{Kuhn} \textit{Let $A$ be a finite dimensional algebra over a field. Then any recollement of $\Mod{A}$ by module categories is equivalent to a recollement induced by an idempotent element.}\\

Indeed, this statement is true for any ring $A$.

\begin{thm}\label{bijection Mod}
A recollement of $\Mod{A}$ is a recollement by module categories if and only if it is equivalent to a recollement induced by an idempotent element of a ring $S$, Morita equivalent to $A$.
\end{thm}
\begin{proof}
By Proposition $5.1$, any recollement of $\Mod{A}$ is equivalent to 
\begin{equation}\label{rec induced by I}
\xymatrix@C=0.5cm{
\Mod{A/I} \ar[rrr]^{\inc} &&& \Mod{A} \ar[rrr]^{j^*} \ar
@/_1.5pc/[lll]_{-\otimes_AA/I}  \ar
 @/^1.5pc/[lll]^{\Hom_A(A/I,-)} &&& \Ccal_I,
\ar @/_1.5pc/[lll]_{j_!} \ar
 @/^1.5pc/[lll]^{j_*}
 } 
\end{equation}
for some idempotent ideal $I$ of $A$ and $\Ccal_I$ the corresponding quotient category.  
Clearly, if $I$ is generated by an idempotent element $e$ in $A$, then $\Ccal_I$ is equivalent to $\Mod{eAe}$, see Example \ref{example}. 

Conversely, assume that (\ref{rec induced by I}) is equivalent to a recollement by module categories. Let $P$ be a small (i.e., $\Hom_{\Ccal_I}(P,-)$ commutes with coproducts) projective generator of $\Ccal_I$, which exists since we assume that $\Ccal_I$ is equivalent to a module category. Let us denote by $C$ the ring $\End_{\Ccal_I}(P)$ and by $\Theta$ the equivalence $\Hom_{\Ccal_I}(P,-)\colon \Ccal_I\lxr \Mod{C}$. The object $j_!(P)$ is projective since we have the adjoint pair $(j_!,j^*)$ and the functor $j^*$ is exact.
It is also small since $j^*$ commutes with coproducts and $P$ is small. Since a projective object is small in a module category if and only if it is finitely generated, there is a surjective map $p\colon A^{\oplus n}\lxr j_!(P)$, for some $n$ in $\mathbb{N}$. This surjective map splits since $j_!(P)$ is a projective $A$-module, i.e., there is an injective map $h\colon j_!(P)\lxr A^{\oplus n}$ such that $ph=\iden_{j_!(P)}$. Let $S$ denote the endomorphism ring of $A^{\oplus n}$, i.e., $S=\End_A(A^{\oplus n})$, and let $\Phi:=\Hom_A(A^{\oplus n},-)$ denote the Morita equivalence between $\Mod{A}$ and $\Mod{S}$. Then we have a surjection $\Phi(p)\colon S=\Phi(A^{\oplus n})\lxr \Phi(j_!(P))$ which splits via $\Phi(h)$, i.e., $\Phi(j_!(P))$ is a direct summand of $S$. Moreover, it is precisely generated by the idempotent $\Phi(h)\Phi(p)$ in $\End_S(S)$, which, under the isomorphism $\End_S(S)\cong S$ is identified with $hp$. Denote this element by $e$. Clearly, $eS$ is the image of $\Phi(h)\Phi(p)$ and it is isomorphic to $\Phi(j_!(P))$ in $\Mod{S}$. Since both $j_!$ and $\Phi$ are fully faithful, 
\begin{eqnarray}
C &=& \Hom_{\Ccal_I}(P,P) \ \cong \ \Hom_A(j_!(P),j_!(P)) \nonumber \\
 &\cong& \Hom_S(\Phi(j_!(P)),\Phi(j_!(P))) \ \cong \ \Hom_S(eS,eS) \ \cong \ eSe. \nonumber 
\end{eqnarray}
The last isomorphism is $\alpha\colon eSe\lxr \End_S(eS)$, sending an element in $eSe$ to the endomorphism given by left multiplication with it. This is clearly an injective ring homomorphism. Given an endomorphim $g$ of $eS$, $g$ is given by left multiplication with $g(e)$. Since $g(e)$ lies in $eS$ and $g(e)e=g(e^2)=g(e)$, $g(e)$ lies in $eSe$. Thus, $\alpha$ is surjective.
Now, the functors $\Theta\colon \Ccal_I\lxr \Mod{eSe}$ and $\Phi$ form an equivalence of recollements from $\textsf{R}(\Mod{A/I},\Mod{A},\Ccal_I)$ to $\textsf{R}(\Mod{S/SeS},\Mod{S},\Mod{eSe})$. Indeed, we have natural isomorphisms
\begin{eqnarray}
\Theta j^*(M) &=& \Hom_{\Ccal_I}(P,j^*M) \ \cong \ \Hom_A(j_!(P),M) \nonumber \\
&\cong & \Hom_S(\Phi(j_!(P)),\Phi(M)) \ = \ \Hom_S(eS,\Phi(M)), \nonumber 
\end{eqnarray}
Since the functor $\Hom_S(eS,-)$ is the quotient functor $\Mod{S}\lxr \Mod{eSe}$, $\Phi$ and $\Theta$ form an equivalence of recollements, as wanted.
\end{proof}

Under additional conditions, we can say more about the ideal $I$ of $A$. If $A$ admits the Krull-Schmidt property for finitely generated projective $A$-modules (i.e., $A$ is semiperfect), we can simplify the statement of the theorem. 

\begin{cor}
Let $A$ be a semiperfect ring. Then a recollement of $\Mod{A}$ is a recollement by module categories if and only if the associated idempotent ideal $I$ is generated by an idempotent element.
\end{cor}
\begin{proof}
This follows from the proof of Theorem $5.2$. Let $P$ be a basic (i.e., every indecomposable summand occurs with multiplicity one) small projective generator of $\Ccal_I$. Then $j_!(P)$ is also basic. Since $A$ satisfies the Krull-Schmidt property for projective modules, $j_!(P)$ is a direct summand of $A^{\oplus n}$ if and only if it is a direct summand of $A$. Using the arguments in the proof of Theorem $5.2$ for $S=A$, we see that there is an equivalence of recollements induced by $\Phi$, $\iden_{\Mod{A}}$ and $\Psi$ from the recollement (\ref{rec induced by I}) to the recollement induced by the idempotent element $e$ (see Example \ref{example}). Thus, the essential image of the embeddings $\Mod{A/I}\lxr \Mod{A}$ and $\Mod{A/AeA}\lxr \Mod{A}$ coincide. By Theorem \ref{Gab de la}, the epimorphisms $f\colon A\lxr A/I$ and $g\colon A\lxr A/AeA$ must then lie in the same epiclass, i.e., there is an isomorphism $h\colon A/I\lxr A/AeA$ such that $hf=g$. Note now that, since $h$ is an isomorphism, we have that $g(I)=0$ and $f(AeA)=0$, thus showing that $I=AeA$, as wanted. 
\end{proof}

Recall that $A$ is semiprimary if the Jacobson radical $\textsf{J}(A)$ is nilpotent and $A/\textsf{J}(A)$ is semisimple. Indeed, semiprimary rings are semiperfect (see, for example, \cite[Corollary 23.19]{Lam}) and  every idempotent ideal is generated by an idempotent element of $A$ (\cite{DRi}). Finite dimensional algebras over a field are well-known examples of semiprimary rings. The following corollary provides an answer to Kuhn's question in the context where it appeared.

\begin{cor}
Let $A$ be a semiprimary ring. Then any recollement of $\Mod{A}$ is equivalent to a recollement induced by an idempotent element of $A$. In particular, any recollement of $\Mod{A}$ is a recollement by module categories.
\end{cor}
\begin{proof}
Let $I$ be the idempotent ideal associated to a recollement of $\Mod{A}$, as in the proof of Theorem \ref{bijection Mod}. Since $A$ is semiprimary, $I=AeA$ for some $e$ idempotent element of $A$. Thus, the equivalent recollement (\ref{rec induced by I}) is induced by the ring epimorphism $f\colon A\lxr A/AeA$, thus finishing the proof.
\end{proof}

\end{document}